\documentclass[10pt]{article}
\usepackage{amsmath}
\usepackage{amsfonts}
\usepackage{amssymb}
\usepackage{graphicx}
\newtheorem{definition}{Definition}%[chapter]
\newtheorem{theorem}{Theorem}%[chapter]
\newtheorem{corollary}{Corollary}%[chapter]
\newtheorem{lemma}{Lemma}%[chapter]
\newtheorem{remark}{Remark}%[chapter]
\newenvironment{proof}{\noindent {\bf Proof:}}

\begin{document}
\title{An optimal control problem for the
Navier-Stokes-$\alpha$ system }
\author{Exequiel Mallea-Zepeda$^1$, Elva Ortega-Torres$^2$, \'Elder J. Villamizar-Roa$^3$}
\date{\small$^1$\it Departamento de Matem\'atica, Universidad de Tarapac\'a, Arica, Chile\\
$^{2}$\it Departamento de Matem\'aticas, Universidad Cat\'olica
del Norte, Antofagasta, Chile\\
$^3$\it Escuela de Matem\'aticas, Universidad Industrial de Santander, Bucaramanga, Colombia}
\maketitle

\footnotetext{$^1$  E-mail:{\tt emallea@uta.cl}}
\footnotetext{$^2$E-mail: {\tt eortega@ucn.cl}}
\footnotetext{$^3$E-mail: {\tt jvillami@uis.edu.co}}
%\maketitle
\date{}

\begin{abstract}
In this paper we study a distributed optimal control problem for a
three-dimensional Navier-Stokes-$\alpha$ model.  We prove the solvability of the optimal control problem,
and derive first-order optimality conditions by using a Lagrange multipliers Theorem. Finally, considering a velocity tracking control problem for the
three-dimensional Navier-Stokes-$\alpha$ model, we analyze the relation of its optimality system to the corresponding one associated  to the Navier-Stokes model by proving a convergence theorem, which establishes that, as the length scale $\alpha$ goes to zero, the optimality system of the three-dimensional Navier-Stokes-$\alpha$ model converges to the optimality system associated with the velocity tracking control problem of the Navier-Stokes equations.

\noindent\textbf{Keywords:} Optimal control problem, $\alpha$-Navier-Stokes model,
optimality conditions.

\noindent\textbf{AMS Subject Classifications (2010):} 49J20, 76D55, 76D05, 35Q30.
%Mathematics Subject Classification (2010) 49J20 · 76D55 · 35Q30

\end{abstract}

\section{Introduction}
The Navier-Stokes-$\alpha$ model (NS-$\alpha$), also known as Lagrange averaged Navier-Stokes-$\alpha$ model, corresponds to a regularization of the Navier-Stokes equations using the Helmholtz operator. This model, introduced by S. Chen, C. Foias, D.D. Holm, E. Oslon, E.S.
Titi, and S. Wynne in \cite{chen1}, modifies the nonlinearity in the Navier-Stokes system  to control the cascading of turbulence at scales smaller than a certain length,
but without introducing any extra dissipation (c.f. \cite{chen1,lans2, chen2, chen3, holm, lans4}).  This model can be deduced as follows: We consider the Navier-Stokes equations which are given by 
\begin{equation}
\left\{
\begin{array}{rcl}%
v_t-\nu \Delta v+(v\cdot\nabla)v+\nabla p&=&f \mbox{ in } Q, \\
\mbox{div}\, v &=&0 \mbox{ in } Q,\\
v&=&0 \mbox{ on } \Gamma \times (0,T),\\
v(x,0)&=& v_0(x) \mbox{ in } \Omega,
\end{array}
\right.  \label{eq0a}
\end{equation}
where $v(x,t)$ and $p(x,t)$ are the unknown, representing respectively, the
 velocity and the pressure, in each point
of $Q=\Omega  \times (0,T), 0 < T <\infty,$ $\Omega$ is a domain of $\mathbb{R}^n,$ $n=2,3,$ with boundary $\Gamma.$ On the
right-hand side, $f$ is a fixed external force, and $v_0$ is a given
initial velocity field. The positive constant $\nu$ represents the
kinematic viscosity of the fluid.
 Then, by using the identity $(v\cdot\nabla)v=-v\times(\nabla\times v)+\frac{1}{2}\nabla(v\cdot v),$ the momentum equation (\ref{eq0a})$_1$ is rewritten as
\begin{equation}
\left\{
\begin{array}{rcl}%
v_t-\nu \Delta v-v\times(\nabla\times v)+\nabla p^{\prime}&=&f \mbox{ in } Q, \\
\mbox{div}\,v &=&0 \mbox{ in } Q,
\end{array}
\right.  \label{eq0ab}
\end{equation}
with $p^{\prime}=p+\frac{1}{2}v\cdot v.$ Therefore, applying the so-called Leray regularization in the nonlinear term of (\ref{eq0ab})$_1$ we have
\begin{equation}
\left\{
\begin{array}{rcl}%
v_t-\nu \Delta v-u\times(\nabla\times v)+\nabla p^{\prime}&=&f \mbox{ in } Q, \\
\mbox{div}\,v &=&0 \mbox{ in } Q,
\end{array}
\right.  \label{eq0ac}
\end{equation}
where $u$ is defined as the solution of
\begin{equation}
\left\{
\begin{array}{rcl}%
u-\alpha^2\Delta u+\nabla \pi&=&v, \\
\mbox{div}\, u &=&0,\\
 u&=&0 \mbox{ on } \Gamma,
\end{array}
\right.  \label{eq0ad}
\end{equation}
with $\alpha^2>0$ being the regularization parameter. One may rewrite (\ref{eq0ac}) in terms of $u$ by replacing $v$ in (\ref{eq0ac}), obtaining the system
\begin{equation}
\left\{
\begin{array}{rcl}
\partial_t(u-\alpha^2 \Delta u)-\nu \Delta(u-\alpha^2 \Delta u)
-u\times (\nabla\times (u-\alpha^2 \Delta u)) + \nabla p''&=& f \mbox{ in } Q, \\
\mbox{div}\, u&=&0 \mbox{ in } Q,
\end{array}
\right.  \label{eq1z}
\end{equation}
where $p''=p^\prime+\partial_t\pi+\Delta \pi$ (here we have used that $\nabla\times\nabla \pi=0$). Since system (\ref{eq1z}) if of fourth order, it needs to be completed with an extra boundary condition for $\Delta u.$ We could consider the homogeneous Dirichlet boundary conditions $u=0$ and $\Delta u=0$ on $\Gamma \times (0,T);$ however, these assumptions are incompatible due the incompressibility condition (see \cite{Ladyzhenskaya,MS1}). Therefore, it is convenient to complete (\ref{eq1z}) with the boundary conditions $u=Au=0$ on $\Gamma \times (0,T),$ where $A$ denotes the Stokes operator. Equations (\ref{eq0ac})-(\ref{eq0ad}) constitutes the so-called Navier-Stokes-$\alpha$ model. Observe that, considering formally $\alpha=0$, we recover the Navier-Stokes system.\\

The main reason of studying the NS-$\alpha$ models comes from the need of approximating problems relating to turbulent flows, because this kind of models preserves properties of transport for circulation and vorticity dynamics of the Navier-Stokes equations. In addition,  the interest of using the NS-$\alpha$ models is justified due to the high-computational cost that the
Navier-Stokes model requires \cite{lans2}. For a complete description of the physical significance of the NS-$\alpha$ models, namely in turbulence theory, and their developments, we refer \cite{chen1,lans2,chen2, chen3,holm,lans4,lans7,lans8} and references therein.\\

From a mathematical point of view, several results devoted to the analysis of NS-$\alpha$ models have been developed in the last years, see for
instance \cite{lans2, lans4, lans7, lans8,real1,marquez-duran, lans5, Cheskidov, lans3,coutand, Fo-Ho-Ti,lans9} and references
therein. These results are related to the
well-posedness, long time behavior, decay rates of the velocity and
the vorticity, the connection between the solutions of the
NS-$\alpha$ model and the 3D Navier-Stokes system, the existence and uniqueness of solutions for
stochastic versions, and the existence and convergence of trajectory attractors, among others. In particular, unlike the 3D Navier-Stokes
equations, for NS-$\alpha$ model, the existence and uniqueness of
weak solutions is known (see for instance \cite{lans7}). In control problems this point
is important because it guarantees that
the reaction of the flow produced by the action of a control is
unique.\\

In this paper we are interested in an optimal control problem for
the NS-$\alpha$ model (\ref{eq1z}). We consider a distributed control acting as a external force; we also allow a final observation in the control; in this sense,
we say that it is a distributed optimal control problem with final observation. More precisely, we wish to
minimize the functional
\[J(u,f)=\frac{\gamma_u}{2}\int_0^T\Vert u(t)-u_d(t)\|_{D(A)}^2
dt+\frac{\gamma_T}{2}\int_\Omega\vert u(x,T)-u_T(x)\vert^2dx
+\frac{\gamma_f}{2}\int_0^T \|f(t)\|_{2}^2 dt,\] where the velocity
field is subject to verify system (\ref{eq1z}), and the field $f$ now represents a distributed type control.
The fields $u_d, u_T$ are given and denote the desired states, and the parameters
$\gamma_u,\gamma_T,\gamma_f\ge0$ stand the cost coefficients for the states and
control. The exact mathematical formulation will be given in Section
3. We will prove the solvability of the optimal control problem and
state the first-order optimality conditions. By using a Lagrange
multipliers theorem, we derive an optimality system. To the best of our knowledge, 
the analysis of optimal control problems where the state
variable satisfies the 3D NS-$\alpha$ model (\ref{eq1z}) has not been considered.
However, from the point of view of the controllability theory, in \cite{Cara} the authors deals with the distributed and boundary controllability for the NS-$\alpha$ model and prove that the Leray-$\alpha$ equations are locally null controllable, with controls bounded independently of $\alpha.$\\

In the context of nonstationary Navier-Stokes equations, there are many
results available in the literature concerned with the study of
optimal control problems (see \cite{fursikov} and references
therein). In particular, for the
2D-Navier-Stokes system, necessary conditions of optimality can be found in \cite{abergel, gunzburger1,
gunzburger2, hinze, Mallea}. Necessary conditions of optimality for control problems related to 3D
Navier-Stokes system were obtained in \cite{casas, Casas2}. In \cite{Casas2}, the author studied a velocity tracking control problem associated to the non-stationary Navier-Stokes equations for three-dimensional flows. In the classical tracking control problem, the cost functional involves the $L^2$-norm of 
$u-u_d,$ but unlike the $2D$ case, the $3D$ version is much more complicated due to the lack of uniqueness of weak solutions, or the existence of strong solutions
(which is an open question). Therefore, instead of considering the $L^2$-norm of the cost functional, in \cite{Casas2} the authors considered
\begin{equation}\label{fo1}
J_0(u,f):=\displaystyle\frac{\gamma_u}{2}\int_0^T\|u(t)-u_d(t)\|^8_{L^4}dt+\displaystyle\frac{\gamma_T}{2}\int_\Omega|u(x,T)-u_T(x)|^2dx
+\displaystyle\frac{\gamma_f}{2}\int_0^T\|f(t)\|^2dt.
\end{equation}
Then, it is possible to minimize $J_0$ in a class of functions which $(u,f)$ satisfies the Navier-Stokes system (\ref{eq0a}). Indeed, if $u$ is a weak solution of (\ref{eq0a}) such that $J_0(u,f)<\infty,$ then $u$ is a strong solution. With this formulation, the authors in \cite{Casas2} proved that there exists an optimal solution and analyzed first and second optimality conditions (see, also \cite{casas}). In this paper we also are interested in to analyze the convergence of the optimality system of the optimal control problem, associated to the N-S-$\alpha$ system as $\alpha\rightarrow 0^+,$ and relate the limit to the corresponding optimality system of the optimal control problem with state equations (\ref{eq0a}) and cost functional (\ref{fo1}). 
In \cite{lans7} the authors investigated the convergence, as $\alpha\rightarrow 0^+,$ of the solutions of the Navier-Stokes-$\alpha$ equations to a weak solution of Navier-Stokes system (\ref{eq0a}). Therefore, inspired in \cite{lans7}, we will analyze the convergence, as $\alpha\rightarrow 0^+,$ of the adjoint system associated to the optimal control problem for N-S-$\alpha$ model, and its relation with the corresponding adjoint system in the case of Navier-Stokes equations. This fact, gives a way to analyze optimal control problems associated with the Navier-Stokes equations, via optimal control problems with state equations given by the Navier-Stokes-$\alpha$ model.\\

The paper is organized as follows. In Section 2, we establish the
notation to be used and recall some preliminary results for the
NS-$\alpha$ model. In Section 3, we are setting the precise optimal
control problem and prove the existence of optimal solutions. In
Section 4, we derive the first-order optimality conditions, and by
using a Lagrange multipliers theorem in Banach spaces, we derive an optimality system. Finally, in Section 5, we analyze the relationship between the optimality systems of NS-$\alpha$ and Navier-Stokes models.
\section{Preliminaries}
Let $\Omega$ be a bounded domain in $\mathbb{R}^3$ with boundary
$\Gamma$ of class $C^2$. We denote by $\mathcal{D}(\Omega)$ the space of functions of class
$C^{\infty}(\Omega)$ with compact support on $\Omega.$ Throughout this paper we,
use standard notations for Lebesgue and Sobolev spaces. In
particular, the $L^2(\Omega)$-norm and the $L^2(\Omega)$-inner
product, will be represented by $\Vert \cdot\Vert$ and
$(\cdot,\cdot),$ respectively. %We consider the space $H_0^1(\Omega)$ with inner product given by $(u,v)+\alpha^2(\nabla u,\nabla v),$ for $u,v\in H_0^1(\Omega),$ where its associated norm $\Vert \cdot \Vert_{H^1}$ is in fact equivalent to the gradient norm, thanks the Poincaré inequality. 
We consider the solenoidal Banach
spaces $H$ and $V$ defined, respectively, as the closure in
$(L^2(\Omega))^3$ and $(H^1(\Omega))^3$ of
$
\mathcal{V}=\{u \in (\mathcal{D}(\Omega))^3 : \mbox{div}\ u=0
\mbox{ in } \Omega\}.
$
The
norm and the inner product in $V$ will be denoted by $\Vert
u\Vert_V$ and $(\nabla u, \nabla v),$ respectively. Throughout this
paper, if $X$ is a Banach space with topological dual space $X'$, the duality
pairing between $X'$ and $X$ will be denoted by $\langle\cdot,
\cdot\rangle_{X',X}.$ To simplify the notation, we will
use the same notation for vectorial valued and scalar valued spaces.
For $X$ Banach space, $\|\cdot\|_X$ denotes its norm and
$L^p(0,T;X)$ denotes the standard space of functions from $[0,T]$ to
$X,$ endowed with the norm
\[\|u\|_{L^p(0,T;X)}= \bigg( \int_0^T\|u\|_X^p dt\bigg)^{1/p}, \ 1\leq p <\infty,\qquad
 \|u\|_{L^\infty(0,T;X)}= \sup_{t\in (0,T)} \|u(t)\|_X.\]
In the sequel we will identify the spaces $L^p(0,T;X):= L^p(X)$ and
$L^p(0,T;L^p(\Omega)) :=L^p(Q)$. Let us consider the Leray projector $P:L^2(\Omega) \rightarrow H$, and denote
by $A:= - P\Delta$ the Stokes operator with domain
$D(A)=H^2(\Omega)\cap V$. It is well-known that $A$ is a self-adjoint
positive operator with compact inverse. Since $\Gamma$ is of class
$C^2$, the norms $\|A u\|$ and $\|u\|_{H^2}$ are equivalent.
Also, for $u\in D(A)$ and $v\in L^2(\Omega),$ and considering the space
$H^{-1}(\Omega)\equiv (H^1_0(\Omega))',$ we define
$
\langle(u\cdot\nabla)v,w\rangle_{H^{-1},H^1_0}=\sum_{i,j=1}^3\langle\partial_i
v_j,u_iw_j\rangle_{H^{-1},H^1_0},\ \forall w\in H^1_0(\Omega).
$
In particular, if $v\in H^1(\Omega),$ the duality product
$\langle(u\cdot\nabla)v,w\rangle_{H^{-1},H^1_0}$ coincides with the
definition of
$$
((u\cdot\nabla)v,w)=\sum_{i,j=1}^3\int_\Omega(u_i\partial_i
v_j)w_jdx.
$$
Let us denote by $(\nabla u)^*$ the transpose of $\nabla u.$ Thus,
if $u \in D(A)$ then $(\nabla u)^* \in H^1(\Omega) \subset
L^6(\Omega)$. Consequently, for $v\in L^2(\Omega)$ we have that
$(\nabla u)^*\cdot v \in L^{3/2}(\Omega)\subset H^{-1}(\Omega)$ and
$$
\langle(\nabla u)^*\cdot
v,w\rangle_{H^{-1},H^1_0}=\sum_{i,j=1}^3\int_\Omega(\partial_ju_i)
v_iw_jdx,\ \forall w\in H^1_0(\Omega).
$$
One can check that for $u,w \in D(A), v\in L^2(\Omega),$ the
following equality holds
\begin{equation}\label{eq7}
\langle(u\cdot \nabla) v, w\rangle_{H^{-1},H^1_0}= -\langle(\nabla
w)^*\cdot v,u\rangle_{H^{-1},H^1_0}.
\end{equation}
We consider the nonlinear operator  $ B: D(A)\times D(A)\rightarrow
D(A)'$ defined by
\begin{equation}\label{eq8}
\langle B(u,v),w\rangle_{D(A)',D(A)}=\langle(u\cdot \nabla)(v-\alpha^2
\Delta v), w\rangle_{V',V}+\langle(\nabla u)^*\cdot (v-\alpha^2\Delta
v),w\rangle_{V',V}.
\end{equation}
Thus, from (\ref{eq7}) we have
\begin{equation}\label{eq9}
\langle B(u,v),u \rangle_{D(A)',D(A)}= 0, \ \forall \, u,v \in D(A).
\end{equation}
Also, we get
\begin{eqnarray*}
|\langle B(u,v),w\rangle_{D(A)',D(A)}|&\leq &C\|u\|\|\nabla v\|\|A
w\|+ C\,\alpha^2 (\|u\|_{L^6}\|\nabla w\|_{L^3}+
\|\nabla u\|\|w\|_{L^\infty})\|\Delta v\|\\
&\leq & C\|\nabla u\|\|A v\|\|A w\|+C\alpha^2 \|\nabla u\|\|A w\|\|Av\|\leq  C\|\nabla u\|\|A v\|\|A w\|.
\end{eqnarray*}
Therefore,
\begin{equation}\label{eq10}
\|B(u,v)\|_{D(A)'} \leq C\,\|\nabla u\|\|A v\|\leq C\Vert
u\Vert_V\Vert v\Vert_{D(A)}, \quad \forall \, u, v \in D(A),
\end{equation}
and thus, for all $u, v \in L^\infty(V)\cap L^2(D(A))$ it holds
$B(u,v) \in  L^2(D(A)')$. 
Denoting by $ \Delta_\alpha= I-\alpha^2 \Delta$, one gets
\[\Delta_\alpha u \in L^\infty(V') \cap L^2(H)\quad \mbox{ and }
\quad \Delta_\alpha Au \in L^2(D(A)') \quad \forall u\in
L^2(D(A))\cap L^\infty(V).\] With the above notations, system
(\ref{eq1z}) can be rewritten as
\begin{equation}
\left\{
\begin{array}{rcl}%
\Delta_\alpha u_t +\nu \Delta_\alpha Au+ B(u,u) + \nabla p
&=& f \mbox{ in } Q,\\
{\rm div}\,u&=&0  \mbox{ in } Q, \\
u&=&0, \quad Au=0 \mbox{ on } \Gamma \times (0,T),\\
u(x,0)&=& u_0(x) \mbox{ in } \Omega.
\end{array}
\right.  \label{eq11}
\end{equation}
%For $f\in L^2(D(A)')$ and the regularity $L^2(D(A)')$ of the terms
%$B(u,u)$ and $nu \Delta_\alpha Au,$ the equation
Now we are in position to establish the definition of weak solution of problem (\ref{eq1z}) (equivalently (\ref{eq11})).
\begin{definition}(\textit{Weak solution})
Let $f\in L^2(Q)$ and $u_0\in V.$ We say that the field $u$ is a \textit{weak solution} of the
problem (\ref{eq11}) if
\begin{equation}\label{space_W}
u\in\mathbb{W}:=\{u\in L^2(D(A))\cap L^\infty(V)\,:\, u_t\in L^2(H)\}
\end{equation}
and satisfies the following variational formulation
\begin{equation}
\left\{
\begin{array}{rcl}%
(u_t, w) +\alpha^2(\nabla u_t, \nabla w)
+\nu (Au,w+\alpha^2  Aw)
+\langle B(u,u), w\rangle_{D(A)',D(A)} 
&=&(f, w), \quad \forall \, w\in D(A),\\
u(x,0)&=& u_0(x) \mbox{ in } \Omega,
\end{array}
\right.  \label{eq16}
\end{equation}
or equivalently,
\begin{equation}
\left\{
\begin{array}{rcl}%
\Delta_\alpha u_t+\nu \Delta_\alpha Au+ B(u,u)
&=& f \mbox{ in } D(A)', \\
A u&=&0 \mbox{ on } \Gamma \times (0,T),\\
u(x,0)&=& u_0(x) \mbox{ in } \Omega.
\end{array}
\right.  \label{eq17}
\end{equation}
\end{definition}
We recall the following
compactness result:
%\begin{lemma} (\cite{lions})\label{lema1}
%Let $B_0, B$ and $B_1$  reflexives Banach spaces with $B_0\subset
%B\subset B_1$ continuously and  $B_0\hookrightarrow B$ compact. For
%$p_0, p_1 \in (1,\infty)$ and $T <\infty$ it is defined the Banach
%space
%\begin{equation}\label{eq5}
%W=\{ u \in L^{p_0} (B_0), \ u_t\in L^{p_1} (B_1)\}.
%\end{equation}
%Then, $W \hookrightarrow L^{p_0}(B)$ compactly and $W \subset C([0,T];B)$.\\
%\end{lemma}
\begin{lemma} (\cite{Simon})\label{lema1}
Let $B_0, B$ and $B_1$  be Banach spaces with $B_0\hookrightarrow B\hookrightarrow
B_1$ continuously and  $B_0\hookrightarrow B$ compact. For $1\leq p\leq
\infty$ and $T <\infty$ consider the Banach space
\begin{equation}\label{eq5}
W=\{ u \in L^{p} (0,T;B_0)\,:\, \ u_t\in L^{1} (0,T;B_1)\}.
\end{equation}
Then $W \hookrightarrow L^p(0,T;B)$ compactly.
\end{lemma}
\begin{remark}\label{compact-injection}
Since
$D(A)\hookrightarrow V \hookrightarrow H$, and $D(A)\hookrightarrow V $ compactly, from Lemma \ref{lema1}
we have  that the injection $\mathbb{W}\hookrightarrow L^2(V)$   is compact; furthermore, using that $D(A),V,H$ are Hilbert spaces, we have  the compact injection
$\mathbb{W} \hookrightarrow C([0,T];V)$ (cf. \cite{lions}). 
\end{remark}

%Below, we give the definition of weak solution and a theorem on the
%existence and uniqueness.
%\begin{definition} (\textit{Weak solution})
%A weak solution of (\ref{eq11}) is a function $u\in L^2(D(A)) \cap
%L^\infty(V)$ with $u_t \in L^2(H)$ satisfying the weak formulation
%(\ref{eq15})-(\ref{eq16}).
%\end{definition}

\begin{theorem}\label{teor2} (\textit{Existence and uniqueness of weak solution}) Assuming that
$f\in L^2(Q)$ and $u_0 \in V$, there exists a unique weak solution
of (\ref{eq11}). Moreover, there exists a positive constant $K:=K(\nu,\alpha,\|u_0\|_{V},\|f\|_{L^2(Q)})$ such that 
\begin{equation}\label{bound_solution}
\|u\|_{\mathbb{W}}\le K.
\end{equation}
%which also verifies that $u\in C([0,T];V)$.
\end{theorem}
\textit{Proof.} The existence of weak solutions follows from the classical Galerkin
approximations and energy estimates \cite{lans2,lans7,real1,lans5,lans3,VRR}; for that, let $\{w_k\}_{k=1}^\infty$ the orthonormal basis of $H$ consisting of eigenfunctions of the Stokes operator $A.$ For each $k\geq 1,$ we consider the vector space $H_k$ spanned by $\{v_1,...,v_k\},$ and $P_k$ be the $L^2$-orthogonal projection from $H$ onto $H_k.$ Then, the corresponding Galerkin approximation for  
(\ref{eq16}) consists in to find $u_k=\sum_{i=1}^kg_{ik}(t)v_i,$ for some scalar functions $g_{ik},$ $1\leq i\leq k,$ defined on $[0,T],$ such that $u_k$ solves the following system of ordinary differential equations:
\begin{equation}
\left\{
\begin{array}{rcl}%
\frac{d}{dt}\{(u_{k}, w) +\alpha^2(\nabla u_k, \nabla w)\}
+\nu (Au_k,w+\alpha^2  Aw)+\langle B(u_k,u_k), w\rangle_{D(A)',D(A)} 
&=&(P_kf, w), \quad \forall \, w\in H_k,\\
\ u_k(x,0)&=& P_ku_0(x) \mbox{ in } \Omega.
\end{array}
\right.  \label{eq16b}
\end{equation}
By the classical theory of ordinary differential equations, for each $k\geq 1,$ the system (\ref{eq16b}) has a unique solution for an interval of time $[0,T_k].$ If $T_k<T,$ then $\Vert u_k\Vert+\alpha^2\Vert \nabla u_k\Vert$ must tend to $\infty$ as $t$ goes to $T_k;$ then, uniform estimates show that this does not happen and thus $T_m=T$ (cf. \cite[Ch. 3]{temam}). To obtain the {\it a priori} estimates, we take $w=u_k(t)$ in (\ref{eq16b}) and thus, taking into
account (\ref{eq9}), we have
\begin{equation}\label{bo-1}
\frac12\frac{d}{dt}(\|u_k\|^2+\alpha^2\|\nabla u_k\|^2)+\nu\|\nabla u_k\|^2+\nu\alpha^2\|Au_k\|^2=(P_kf,u_k).
\end{equation}
From the H\"{o}lder and Young inequalities, we obtain
$$
(P_kf,u_k)\le \|f\|\|u_k\|\le C\|f\|\|\nabla u_k\|\le C\|f\|^2+\frac\nu2\|\nabla u_k\|^2,
$$
which,  jonitly to (\ref{bo-1}) implies
\begin{equation}\label{bo-2}
\frac{d}{dt}(\|u_k\|^2+\alpha^2\|\nabla u_k\|^2)+\nu\|\nabla u_k\|^2+2\nu\alpha^2\|Au_k\|^2\le C\|f\|^2.
\end{equation} 
Thus, integrating $(\ref{bo-2})$ from $0$ to $t\in[0,T]$, we have
\begin{eqnarray}
\|u_k(t)\|^2+\alpha^2\|\nabla u_k(t)\|^2+\nu\int_0^t(\|\nabla u_k(s)\|^2+2\alpha^2\|Au_k(s)\|^2)ds\le C\int_0^t\|f(s)\|^2ds+\|u_0\|^2+\alpha^2\|\nabla u_0\|^2.\label{bo-3}
\end{eqnarray}
Since $(u_0,f)\in V\times L^2(Q)$, from (\ref{bo-3}) we conclude that there exists a constant
$K_1:=K_1(\nu,\alpha,\|u_0\|_V,\|f\|_{L^2(Q)})$ such that 
\begin{equation}\label{bo-4}
\|u_k\|_{L^\infty(V)}+\|u_k\|_{L^2(D(A))}\le K_1.
\end{equation}
Moreover, from (\ref{eq16b}), for each $w\in H_k$ we deduce
\begin{eqnarray*}
\langle \Delta_\alpha u_{k_t},w\rangle_{D(A)',D(A)}
&=&-\nu(\nabla u_k, \nabla w)-\nu\alpha^2(Au_k,Aw)-\langle B(u_k,u_k),w\rangle_{D(A)',D(A)}+(P_kf,w),
\end{eqnarray*}
and then, from H\"{o}lder inequality, (\ref{eq10}) and (\ref{bo-4}) we obtain
\begin{eqnarray}
|\langle\Delta_\alpha u_{k_t},w\rangle_{D(A)',D(A)}|
&\le&C(\|\nabla u_k\|+\|Au_k\|+\|B(u_k,u_k)\|_{D(A)'}+\|f\|)\|w\|_{D(A)}\nonumber\\
&\le&C(\|\nabla u_k\|+\|Au_k\|+\|\nabla u_k\|\|Au_k\|+\|f\|)\|w\|_{D(A)}\nonumber\\
&\le&K_2(\|\nabla u_k\|+\|Au_k\|+\|f\|)\|w\|_{D(A)},\label{bo-5}
\end{eqnarray}
where $K_2:=K_2(\nu,\alpha,\|u_0\|_V,\|f\|_{L^2(Q)})$. Since $\langle\Delta_\alpha u_{k_t},w\rangle_{D(A)',D(A)}=\langle u_{k_t}+\alpha^2 Au_{k_t},w\rangle_{D(A)',D(A)}$ for 
all $w\in D(A)$, then from (\ref{bo-5}) we have
$$
\|u_t+\alpha^2 Au_{k_t}\|_{D(A)'}\le K_2(\|\nabla u_k\|+\|Au_k\|+\|f\|);
$$
thus
\begin{equation}\label{bo-6}
\|u_{k_t}+\alpha^2 Au_{k_t}\|^2_{D(A)'}\le K_3(\|\nabla u_k\|^2+\|Au_k\|^2+\|f\|^2).
\end{equation}
Integrating (\ref{bo-6}) from $0$ to $t\in[0,T]$ and taking into account (\ref{bo-4}) we obtain
\begin{equation}\label{bo-7}
\int_0^t\|u_{k_t}(s)+\alpha^2 Au_{k_t}(s)\|^2_{D(A)'}ds\le C,
\end{equation}
where $C$ is a constant which depends on $K_1,K_2$, and $K_3$.

On the other hand,  by using that the operator A is self-adjoint and positive, and arguing as \cite[Section 3]{lans3} we get
\begin{equation*}
\|v\|^2_{D(A)'}\le \|v+\alpha Av\|^2_{D(A)'}\ \mbox{ for each }\ v\in D(A)',
\end{equation*}
which implies
\begin{equation}\label{bo-8}
\|\alpha Au_{k_t}\|^2_{D(A)'}\le \|u_{k_t}+\alpha Au_{k_t}\|^2_{D(A)'}+\|u_{k_t}\|^2_{D(A)'}
\le C\|u_{k_t}+\alpha u_{k_t}\|^2_{D(A)'}.
\end{equation}
Therefore, from (\ref{bo-7}) and (\ref{bo-8}) we conclude that $Au_{k_t}\in L^2(D(A)')$; and, in particular, there exists a positive constant $C$
such that
\begin{equation}\label{bo-9}
\|u_{k_t}\|_{L^2(H)}\le C.
\end{equation}
Following a standard compactness procedure, previous estimates allow us to pass to the limit as $k$ goes to $+\infty.$ Also, (\ref{bound_solution}) follows from (\ref{bo-4}) and (\ref{bo-9}).  The uniqueness follows from a classical comparison argument and using the
Gronwall Lemma.

\section{A distributed control problem: Existence of optimal solution}
In this section, we establish the statement of the optimal control problem which we will consider. Let us denote by  $\mathcal{U}$ the admissible control set. We suppose that 
\begin{equation}\label{control_set}
\mathcal{U}\subset L^2(Q)\mbox{ is a nonempty, closed and convex set.}
\end{equation}
We consider initial  data $u_0\in V$, and the function $f\in\mathcal{U}$ describing the distributed control
acting on domain $\Omega$.
Then, we define the following constrained problem related to weak solutions of system (\ref{eq11}): 
\begin{equation}\label{eq24}
\left\{
\begin{array}{l}
\mbox{Find $(u,f)\in \mathbb{W}\times \mathcal{U}$ such that:}\\
J(u,f):=\displaystyle\frac{\gamma_u}{2}\int_0^T\|u(t)-u_d(t)\|^2_{D(A)}+\displaystyle\frac{\gamma_T}{2}\int_\Omega|u(x,T)-u_T(x)|^2dx+\displaystyle\frac{\gamma_f}{2}\int_0^T\|f(t)\|^2dt\\
\mbox{is minimized, subject to $(u,f)$ being a weak solution of  (\ref{eq11}).}
\end{array}
\right.
\end{equation}
Here, the pair $(u_d,u_T)\in D(A)\times H$ represents the desires states and the nonnegative real numbers $\gamma_u$, $\gamma_T$ and $\gamma_f$ measure 
the cost of the states and control, respectively. These numbers are non zero simultaneously. The functional $J$ defined in (\ref{eq24}) describes the
deviation of the velocities field from a desired field $u_d$, and the deviation of the velocities field in the final time $T$ from a desired field $u_T$, plus
the cost of the control $f$ measured in the $L^2$-norm.

The admissible set for the optimal control problem (\ref{eq24}) is defined by

\begin{equation}\label{eq28}
 \mathcal{S}_{ad}=\{(u,f) \in \mathbb{W}\times \mathcal{U}\,:\, (u,f)\mbox{ is a weak solution of } (\ref{eq11})\}.
\end{equation}

\subsection{Existence of Global Optimal Solution}

We will show that the optimal control problem
(\ref{eq24}) has a global optimal solution.
\begin{definition}\label{def_optimal}
A pair $(\hat{u},\hat{f})\in\mathcal{S}_{ad}$ will be called a global optimal solution of problem
(\ref{eq24}) if
\begin{equation}\label{opt-1}
J(\hat{u},\hat{f})=\min_{(u,f)\in\mathcal{S}_{ad}}J(u,f).
\end{equation}
\end{definition}
\begin{theorem}\label{existence-optimal}
Let $u_0\in V$. We assume that either $\gamma_f>0$ or $\mathcal{U}$ is bounded in $L^2(Q).$ Then the extremal problem (\ref{eq24}) has at least one global optimal solution $(\hat{u},\hat{f})\in\mathcal{S}_{ad}$.

\end{theorem}
\begin{proof} 
From Theorem \ref{teor2}, we have that
 $\mathcal{S}_{ad}$ is nonempty.  Let $\{(u^m,f^m)\}_{m\ge 1}\subset\mathcal{S}_{ad}$ be a minimizing
sequence of $J$, that is, $\displaystyle\lim_{m\rightarrow\infty}J(u^m,f^m)=\inf_{(u,f)\in\mathcal{S}_{ad}}J(u,f)$. Then, from definition of $\mathcal{S}_{ad}$,
for each $m\in\mathbb{N}$, $(u^m,f^m)$ satisfies system (\ref{eq16}).

Moreover, from the definition of $J$ and the assumption $\gamma_f>0$ or $\mathcal{U}$ is bounded in $L^2(Q)$, we deduce that
\begin{equation}\label{opt-2}
\{f^m\}_{m\ge 1}\mbox{ is bounded in }L^2(Q).
\end{equation} 
From (\ref{bound_solution}) we deduce that there exists a positive constant $C$, independent of $m$, such that
\begin{equation}\label{opt-3}
\|u^m\|_{\mathbb{W}}\le C.
\end{equation}
Then, from (\ref{opt-2}), (\ref{opt-3}), and taking into account that $\mathcal{U}$ is a closed and convex
subset of $L^2(Q)$ (hence is weakly closed in $L^2(Q)$), we deduce that there exists an element
$(\hat{u},\hat{v})\in\mathbb{W}\times\mathcal{U}$ such that, for some subsequence of
$\{(u^m,f^m)\}_{m\ge 1}$, still denoted by $\{(u^m,f^m)\}_{m\ge 1}$, the following convergences  hold, as $m\rightarrow\infty$:

\begin{equation}\label{opt-4}
\left\{
\begin{array}{rcl}
u^m &\rightarrow& \hat{u} \mbox{ weakly in } \mathbb{W},\\
%u^m_t &\rightarrow& \hat{u}_t \mbox{ weakly in } L^2(Q),\\
%A u^m_t &\rightarrow& A\hat{u}_t\mbox{ weakly in } L^2((D(A)'),\\
f^m &\rightarrow& \hat{f} \mbox{ weakly in } L^2(Q),\, \mbox{ and }\, \hat{f}\in\mathcal{U}.
\end{array}
\right.
\end{equation}
From Remark \ref{compact-injection}, we  have
\begin{equation}\label{opt-4-1}
u^m\rightarrow\hat{u}\mbox{ strongly in }L^2(V)\cap C([0,T];V).
\end{equation}
Moreover, from (\ref{opt-4-1}) we have that $u^m(0)$ converges to $\hat{u}(0)$ in $V$, and since $u^m(0)=u_0$ for all $m$, 
we deduce that $\hat{u}(0)=u_0$. Thus, $\hat{u}$ satisfies the initial condition given in  (\ref{eq16})$_2$. Therefore, considering the convergences  (\ref{opt-4})-(\ref{opt-4-1}), and following a standard argument we can
pass to the limit in (\ref{eq16})$_1$ written  by $(u^m,f^m)$, as $m$ goes to $\infty$, and we conclude that $(\hat{u},\hat{f})$ is a solution of $(\ref{eq16})$.
Consequently $(\hat{u},\hat{f})\in\mathcal{S}_{ad}$ and 
\begin{equation}\label{opt-5}
\lim_{m\rightarrow\infty}J(u^m,f^m)=\inf_{(u,f)\in\mathcal{S}_{ad}}J(u,f)\le J(\hat{u},\hat{f}).
\end{equation}
Also, since $J$ is lower semicontinuous on admissible set $\mathcal{S}_{ad}$, we have
$J(\hat{u},\hat{f})\le \displaystyle\liminf_{m\rightarrow\infty}J(u^m,f^m)$, which jointly to (\ref{opt-5}), implies
(\ref{opt-1}).
 \end{proof}

\section{First-order optimality conditions}
In this section we will derive an optimality system for a local optimal solution $(\hat{u},\hat{f})$ of control  problem (\ref{eq24}). We will base on a generic result given by
Zowe et al. \cite{zowe}  on the existence of Lagrange multipliers in Banach spaces (see, also \cite[Ch. 6]{troltz}). This method has been used by Guill\'en-Gonz\'alez et al.
\cite{guillen-1,guillen-2} in the context of  chemo-repulsion systems.

To introduce the  results given in \cite{zowe} we consider the
following abstract optimization problem:
\begin{equation}\label{p-1}
\min_{x\in\mathcal{M}}J(s)\mbox{ subject to }F(x)=0,
\end{equation}
where $J:\mathbb{X}\rightarrow\mathbb{R}$ is a functional, $F:\mathbb{X}\rightarrow\mathbb{Y}$ is an operator,
$\mathbb{X}$ and $\mathbb{Y}$ are Banach spaces, and $\mathcal{M}\subset\mathbb{X}$ is a nonempty, closed
and convex set. The admissible set for problem (\ref{p-1}) is given by 
$$
\mathcal{S}=\{x\in\mathcal{M}\,:\, F(x)=0\}.
$$
Moreover, we consider the functional $L:\mathbb{X}\times\mathbb{Y}'\rightarrow\mathbb{R}$ given by
\begin{equation}\label{p-2}
L(x,\lambda):=J(x)-\langle\lambda,F(x)\rangle_{\mathbb{Y}',\mathbb{Y}},
\end{equation}
which is called Lagrangian functional related to problem (\ref{p-1}).
\begin{definition}(Lagrange multiplier)\label{multiplier}
Let $\hat{x}\in\mathcal{S}$ be a local optimal solution of (\ref{p-1}). Suppose that $J$ and $F$
are Fr\'echet differentiable in $\hat{x}$, with derivatives denote by $J'(\hat{x})$ and $F'(\hat{x})$, respectively.
Then, $\lambda\in\mathbb{Y}'$ is called Lagrange multiplier for problem (\ref{p-1}) at the point $\hat{x}$ if
\begin{equation}\label{p-3}
\left\{\begin{array}{rcl}
\langle\lambda, F(\tilde{x})\rangle_{\mathbb{Y}'}&=&0,\\
L'(\hat{x},\lambda)[s]&:=&J'(\hat{x})[s]-\langle\lambda,F'(\hat{x})[s]\rangle_{\mathbb{Y}',\mathbb{Y}}\ge 0\quad \forall s\in\mathcal{C}(\hat{x}),
\end{array}
\right.
\end{equation}
where $\mathcal{C}(\hat{x})$ is the conical hull of $\tilde{x}$ in $\mathcal{M}$, that is, 
$\mathcal{C}(\hat{x})=\{\theta(x-\hat{x})\,:\, x\in\mathcal{M},\ \theta\ge 0\}$.
\end{definition}
\begin{definition}\label{regular-point}
Let $\hat{x}\in\mathcal{S}$ be a local optimal solution of problem (\ref{p-1}). We say that $\hat{x}$ is  a regular
point if
\begin{equation}\label{p-4}
F'(\hat{x})[\mathcal{C}(\hat{x})]=\mathbb{Y}.
\end{equation}
\end{definition}
The following result guarantees the existence of Lagrange multiplier for problem (\ref{p-1}); the proof can be found in
\cite[Theorem 3.1]{zowe} and \cite[Theorem 6.3, p. 330]{troltz}.
\begin{theorem}\label{existence-regular}
Let $\hat{x}\in\mathcal{S}$ be a local optimal solution of problem (\ref{p-1}). Suppose that $J$ is Fr\'echet differentiable in $\hat{x}$
and $F$ is continuously Fr\'echet differentiable in $\hat{x}$. If $\hat{x}$ is a regular point, then the set of Lagrange multipliers
for (\ref{p-1}) at $\hat{x}$ is nonempty.
\end{theorem}

Now, we will reformulate the control problem (\ref{eq24}) in the abstract context (\ref{p-1}). We consider the following Banach spaces
\begin{equation}\label{spaces_XY}
\mathbb{X}:=\mathbb{W}\times L^2(Q),\ \mathbb{Y}:=L^2(D(A)')\times V,
\end{equation}
and the operator $F:=(F_1,F_2):\mathbb{X}\rightarrow\mathbb{Y}$, where $F_1:\mathbb{X}\rightarrow L^2(D(A)')$ and 
$F_2:\mathbb{X}\rightarrow V$ are defined in each point $x:=(u,f)\in\mathbb{X}$ by
\begin{equation}\label{p-5}
\left\{
\begin{array}{rcl}
F_1(x)&=&\Delta_\alpha u_t+\nu\Delta_\alpha Au+B(u,u)-f,\\
F_2(x)&=&u(0)-u_0.
\end{array}
\right.
\end{equation}
Taking $\mathcal{M}:=\mathbb{W}\times \mathcal{U}$, the optimal control problem (\ref{eq24}) is reformulated as follows:
\begin{equation}\label{p-6}
\min_{x\in\mathcal{M}}J(x)\ \mbox{ subject to }\ F(x)={ 0}.
\end{equation}

We observe that from Definition \ref{multiplier} it follows that the Lagrangian associated to control problem (\ref{p-6}) is the functional
$L:\mathbb{X}\times\mathbb{Y}'\rightarrow\mathbb{R}$ defined by
\begin{equation}\label{p-7}
L(x,\lambda,\eta)=J(x,\lambda,\eta)-\langle F_1(x),\lambda\rangle_{L^2(D(A)'),L^2(D(A))}-\langle\eta,F_2(x)\rangle_{V',V}.
\end{equation}
Moreover, taking into account that $\mathcal{M}$ is a closed and convex subset of $\mathbb{X}$, we have that the set
of admissible solutions of problem (\ref{p-6}) is
$$
\mathcal{S}_{ad}=\{x=(u,f)\in\mathcal{M}\,:\, F(x)=0\}.
$$
With respect to differentiability of functional $J$ and operator $F$, we have the following results, whose proof is standard. 
\begin{lemma}\label{deriJ}
The functional $J$ is Fr\'echet differentiable and the Fr\'echet drivative of $J$ in  $\hat{x}=(\hat{u},\hat{f})\in\mathbb{X}$ in
the direction ${r}=(w,z) \in \mathbb{X}$ is 
\begin{equation}\label{p-8}
J'(\hat{x})[r]=\gamma_u\int_0^T (Aw, A{\hat u}-Au_d)dt+\gamma_T(w(T),\hat{u}(T)-u_T)
+\gamma_f\int_0^T(\hat{f},z)dt.
\end{equation}
\end{lemma}

\begin{lemma}\label{deriF}
The operator $F$ is continuously Fr\'echet differentiable and the Fr\'echet derivative of $F$ in $\hat{x}=(\hat{u},\hat{f})\in\mathbb{X}$
in the direction $r=(w,z)\in\mathbb{X}$ is the linear and bounded operator $F'(\hat{x})[r]=(F_1'(\hat{x})[r],F_2'(\hat{x})[r])$ defined by

\begin{equation}\label{p-9}
\left\{
\begin{array}{rcl}
F'_1(\hat{x})[r]&=&\Delta_\alpha w_t+\nu \Delta_\alpha Aw +B'(\hat{u},\hat{u})w-z,\\
F'_2(\hat{x})[r]&=&w(0),
\end{array}
\right.
\end{equation}
where $B'(\hat{u},\hat{u})w:=B(\hat{u},w)+B(w,\hat{u})$ is the Fr\'echet derivative of $B$ with respect to $u$ in an arbitrary point $(\hat{u},\hat{u})$.
\end{lemma}

Now, we wish to prove the existence of Lagrange multipliers, which is guaranteed if a local optimal
solution of problem (\ref{p-6}) is a regular point (see Theorem \ref{existence-regular} above).

\begin{remark}\label{rem_regular}
From Definition \ref{regular-point} we conclude that $\hat{x}=(\hat{u},\hat{f})\in\mathcal{S}_{ad}$ is a regular point if for any
$(f_u,w_0)\in\mathbb{Y}$ there exists $r=(w_0,z)\in\mathbb{W}\times\mathcal{C}(\hat{f})$ such that
$$
F'(\hat{x})[r]=(f_u,w_0),
$$
where $\mathcal{C}(\hat{f}):=\{\theta(f-\hat{f})\,:\, \theta\ge0, \, f\in\mathcal{U}\}$ is the conical hull of $\hat{f}$ in $\mathcal{U}$.
\end{remark}
\begin{lemma}\label{lem_regular}
Let $\hat{x}=(\hat{u},\hat{f})\in\mathcal{S}_{ad}.$ Then $\hat{x}$ is a regular point.
\end{lemma}
\begin{proof}
Let $(\hat{u},\hat{f})\in\mathcal{S}_{ad}$ fixed and $(f_u,w_0)\in\mathbb{Y}$. Since 
$0\in \mathcal{C}(\hat{f})=\{\theta(f-\hat{f})\,:\, \theta\ge0,\, f\in\mathcal{U}\}$, it is enough to prove
the existence of $w\in\mathbb{W}$ such that solve the following linear problem
\begin{equation}\label{p-10}
\left\{
\begin{array}{rcl}
\Delta_\alpha w_t+\nu\Delta_\alpha Aw+B'(\hat{u},\hat{u})w&=&f_u,\\
w(0)&=&w_0.
\end{array}
\right.
\end{equation}
The existence of solutions of system (\ref{p-10}) follows from Galerkin approximations and energy estimates, similarly 
as the proof of Theorem \ref{teor2}.
\end{proof}
In the following result, we prove the existence of Lagrange multipliers for optimal control problem (\ref{p-6}) 
related to a local optimal solution $\hat{x}=(\hat{u},\hat{f})\in\mathcal{S}_{ad}$.
\begin{theorem}\label{mult}
Let $\hat{x}=(\hat{u},\hat{f})\in\mathcal{S}_{ad}$ be a local optimal solution for problem (\ref{p-6}). Then, there exists
a Lagrange multiplier $(\lambda,\eta)\in L^2(D(A))\times V'$ such that for all $(w,z)\in \mathbb{W}\times\mathcal{C}(\hat{f})$
the following variational inequality holds
\begin{eqnarray}\label{p-11}
&&\gamma_u\int_0^T(Aw,A\hat{u}-Au_d)dt+\gamma_T(w(T),\hat{u}(T)-u_T)+\gamma_f\int_0^T(\hat{f},z)\,dt\nonumber\\
&&-\int_0^T\langle\Delta_\alpha w_t+\nu\Delta_\alpha Aw+B'(\hat{u},\hat{u})w-z,\lambda\rangle_{D(A)',D(A)}\,dt
-\langle\eta,w(0)\rangle_{V',V}\ge 0.
\end{eqnarray}
\end{theorem}
\begin{proof}
From Lemma \ref{lem_regular}, we have that $\hat{x}=(\hat{u},\hat{f})\in\mathcal{S}_{ad}$ is a regular point. Thus, from Theorem \ref{existence-regular}
we deduce that there exists a Lagrange multiplier $(\lambda,\eta)\in L^2(D(A))\times V'$ such that
\begin{equation}\label{p-12}
L'(\hat{x},\lambda,\eta)[r]=J'(\hat{x})[r]-\langle F_1'(\hat{x}),\lambda\rangle_{L^2(D(A)'),L^2(D(A))}-\langle\eta,F_2'(\hat{x})\rangle_{V',V}\ge 0,
\end{equation}
for all $r=(w,z)\in \mathbb{W}\times\mathcal{C}(\hat{f})$. Therefore, the proof follows from (\ref{p-8}), (\ref{p-9}) and (\ref{p-12}).
\end{proof}

From Theorem \ref{mult} we can derive an optimality system for optimal control problem (\ref{p-6}); for which we consider the following linear space
\begin{equation}\label{p-13}
\mathbb{W}_{u_0}:=\{u\in \mathbb{W}\,:\, u(0)=0\}.
\end{equation}

\begin{corollary}\label{corol-1}
Let $\hat{x}=(\hat{u},\hat{f})\in\mathcal{S}_{ad}$ be a local optimal solution for the optimal control problem (\ref{p-6}). Then, the Lagrange multiplier
$\lambda\in L^2(D(A))$ provided by Theorem \ref{mult} satisfied the adjoint system
\begin{equation}\label{adj-1}
\left\{
\begin{array}{rcl}
-\Delta_\alpha\lambda_t+\nu\Delta_\alpha A\lambda-\hat{u}\cdot\nabla\lambda&+&\alpha^2\Delta(\hat{u}\cdot\nabla\lambda+\lambda\cdot\nabla\hat{u})
-(\nabla\lambda)^*\Delta_\alpha\hat{u}+\alpha^2\lambda\cdot\nabla(\Delta\hat{u})\\
&=&-\gamma_u\Delta A(\hat{u}-u_d)\ \mbox{ in }\ L^2(D(A)'),\\
{\rm div}\,\lambda&=&0\ \mbox{ in }\ Q,\\
\lambda&=&0\ \mbox{ on }\ \Gamma\times(0,T),\\
\Delta_\alpha\lambda(T)&=&\gamma_T(\hat{u}(T)-u_T)\ \mbox{ in }\ \Omega,
\end{array}
\right.
\end{equation}
and the optimality condition
\begin{equation}\label{adj-2-2}
\int_0^T(\gamma_f\hat{f}+\lambda,f-\hat{f})\ge 0\  \ \forall f\in\mathcal{U}.
\end{equation}
\end{corollary}
\begin{proof}
Taking $w=0$ in (\ref{p-11}) we have
\begin{equation}\label{adj-3} 
\gamma_f\int_0^T(\hat{f}+\lambda,z)dt\ge0\  \ \forall z\in\mathcal{C}(\hat{f}).
\end{equation}
Then, choosing $z=f-\hat{f}\in \mathcal{C}(\hat{f})$, for all $f\in\mathcal{U}$ in (\ref{adj-3}) we obtain (\ref{adj-2-2}).

Now, we will derive system (\ref{adj-1}). Indeed, taking $z=0$ in (\ref{p-11}) and using that $\mathbb{W}_{u_0}$ is a vector space, we have
\begin{equation}\label{ad}
\int_0^T\langle\Delta_\alpha w_t+\nu\Delta_\alpha Aw+B'(\hat{u},\hat{u})w,\lambda\rangle_{D(A)',D(A)}dt\nonumber\\
=\gamma_u\int_0^T(Aw,A\hat{u}-Au_d)dt+\gamma_T(w(T),\hat{u}(T)-u_T)\ \forall w\in\mathbb{W}_{u_0}.
\end{equation}
Integrating by parts in $\Omega$, we have
\begin{eqnarray}
\langle\Delta_\alpha w_t,\lambda\rangle_{D(A)',D(A)}&=&(w_t,\lambda)+\alpha^2(\nabla w_t,\nabla\lambda)=(w_t,\lambda)-\alpha^2(w_t,\Delta\lambda)
=(\Delta_\alpha\lambda,w_t),\label{ad-1}\\
\langle\nu\Delta_\alpha Aw,\lambda\rangle_{D(A)',D(A)}&=&\nu(Aw,\lambda)-\alpha^2\nu\langle\Delta Aw,\lambda\rangle_{D(A)'}=\nu(Aw,\lambda)-\alpha^2\nu(Aw,\Delta\lambda)\nonumber\\
&=&\langle\nu A\Delta_\alpha\lambda,w\rangle_{D(A)',D(A)}\label{ad-2}.
\end{eqnarray}
Taking into account that $A=-P\Delta$, we obtain
\begin{equation}\label{ad-3}
\gamma_u\int_0^T(Aw,A\hat{u}-Au_d)dt=-\gamma_u\int_0^T(\Delta w,A(\hat{u}-u_d))dt=-\gamma_u\int_0^T\langle\Delta A(\hat{u}-u_d),w\rangle_{D(A)',D(A)}dt.
\end{equation}
Since $B'(\hat{u},\hat{u}):\mathbb{W}_0\rightarrow L^2(D(A)')$ and $\mathbb{W}_0\subset\mathbb{W}$, then  the adjoint operator of $B'(\hat{u},\hat{u})$ is given by
\begin{equation}\label{ad-4}
\langle B'^*(\hat{u},\hat{u}\lambda,w)\rangle_{\mathbb{W}_0',\mathbb{W}}:=\langle B'(\hat{u},\hat{u})w,\lambda\rangle_{L^2(D(A)'),L^2(D(A))}.
\end{equation}

Then, by replacing (\ref{ad-1})-(\ref{ad-3}) in (\ref{ad}) and taking into account (\ref{ad-4}), we obtain
\begin{eqnarray}
\langle\Delta_\alpha\lambda,w_t\rangle_{L^2(D(A)'),L^2(D(A))}&=&-\langle\nu A\Delta_\alpha\lambda,w\rangle_{L^2(D(A)'),L^2(D(A))}-\langle B'^*(\hat{u},\hat{u})\lambda,w\rangle_{\mathbb{W}_0',\mathbb{W}_0}\nonumber\\
&=&-\gamma_u\langle\Delta A(\hat{u}-u_d),w\rangle_{L^2(D(A)'),L^2(D(A))}+\gamma_T(w(T),\hat{u}(T)-u_T).\label{ad-5}
\end{eqnarray}
In order to obtain a representation of the weak derivative in time of $\Delta_\alpha\lambda$ we will analyze the regularity of $B'^*(\hat{u},\hat{u})\lambda$. Indeed, notice that from
(\ref{eq8}) and (\ref{eq7}) we have
\begin{eqnarray}\label{eq61}
\hspace{-0.8cm}\langle B'(\hat{u},\hat{u})w,\lambda
\rangle_{D(A)',L^2(D(A))}&=&\langle\hat{u}\cdot\nabla \Delta_\alpha w,
\lambda\rangle_{V',V}-\alpha^2((\nabla \hat{u})^*\cdot \Delta w, \lambda)\nonumber \\
&&+\langle w\cdot\nabla\Delta_\alpha\hat{u}, \lambda\rangle_{V',V}
-\alpha^2((\nabla w)^*\cdot \Delta\hat{u}, \lambda)\nonumber\\
&=&-(\hat{u}\cdot\nabla \lambda,\Delta_\alpha w)
-\alpha^2(\lambda \cdot\nabla \hat{u}, \Delta w)-(w\cdot\nabla \lambda,\Delta_\alpha \hat{u}) -\alpha^2(\lambda\cdot\nabla w, \Delta\hat{u}).
\end{eqnarray}
We will bound the terms in (\ref{eq61}). From H\"older and Sobolev inequalities we obtain
\begin{eqnarray}
|(\hat{u}\cdot\nabla \lambda,\Delta_\alpha w)|&\leq &
C\|\hat{u}\|_{L^6}\|\nabla \lambda\|_{L^3}\|\Delta_\alpha w\|
\leq C\|\hat{u}\|_V \|\lambda\|_{D(A)} \|w\|_{D(A)},\label{eq65}\\
|(\lambda \cdot\nabla \hat{u}, \Delta w)| &\leq &
C\|\lambda\|_{L^\infty}\|\nabla \hat{u}\|\| \Delta w\| \leq
C\|\lambda\|_{D(A)}\|\hat{u}\|_V \|w\|_{D(A)}.\label{eq66}
\end{eqnarray}
By observing that $ h\cdot \nabla v =0$ on $\Gamma$ if $(h,v) \in D(A)\times D(A)$, and using integration by parts on $\Omega$, for $h, v, z \in
D(A)$ we have
\begin{equation}\label{eq67}
(h\cdot \nabla v, \Delta z)= (\nabla (h\cdot \nabla v), \nabla z)
=(\nabla v  \nabla h, \nabla z) + (h  \nabla(\nabla v), \nabla z),
\end{equation}
where $h  \nabla(\nabla v)=\sum_{i=1}^3h_i\frac{\partial}{\partial
x_i}\nabla v.$ Then, from (\ref{eq67}), the fact that $\|\nabla
v\|_{L^4}\leq C\|v\|_{D(A)}$ and $D(A) \subset L^\infty(\Omega)$, we
obtain
\begin{eqnarray}
|(w\cdot\nabla \lambda,\Delta_\alpha \hat{u})|&=&
|(w\cdot\nabla \lambda,\hat{u})-\alpha^2 (w\cdot\nabla \lambda,\Delta \hat{u})|
\leq  |(w\cdot\nabla \lambda,\hat{u})| +\alpha^2 |(\nabla\lambda\cdot\nabla w,\nabla \hat{u})|
+\alpha^2 |(w\cdot \nabla(\nabla \lambda),\nabla \hat{u})|\nonumber \\
&\leq & C\|w\|_{L^3}\|\nabla \lambda\|\|\hat{u}\|_{L^6} +C(\|\nabla\lambda\|_{L^4}\|\nabla h\|_{L^4}
+\|w\|_{L^\infty}\| \nabla(\nabla \lambda)\|)\|\nabla \hat{u}\|\nonumber \\
&\leq &C \|h\|_{D(A)}\|\lambda\|_{D(A)}\|\hat{u}\|_V,\label{eq68}\\
|(\lambda\cdot\nabla w,\Delta\hat{u})| &\leq & |(\nabla w\cdot\nabla
\lambda,\nabla \hat{u})|
+|( \lambda\cdot\nabla(\nabla w),\nabla \hat{u})|
\leq  C(\|\nabla w\|_{L^4}\|\nabla \lambda\|_{L^4}
+\|\lambda\|_{L^\infty}\| \nabla(\nabla w)\|)\|\nabla \hat{u}\|\nonumber\\
&\leq &C \|w\|_{D(A)}\|\lambda\|_{D(A)}\|\hat{u}\|_V.\label{eq69}
\end{eqnarray}
From (\ref{ad-4}), (\ref{eq61})-(\ref{eq66}), (\ref{eq68}) and
(\ref{eq69}), and by using the H\"older inequality, for $\lambda \in
L^2(D(A))$, $\hat{u}\in \mathbb{W}$ and $w \in \mathbb{W}_0$, we have
$
|\langle B'^*(\hat{u}, \hat{u})\lambda,
h\rangle_{\mathbb{W}'_0,\mathbb{W}_0}|\leq C
\|\hat{u}\|_{L^\infty(V)} \|\lambda\|_{L^2(D(A))}\|w\|_{L^2(D(A))},
$
which implies
\begin{equation}\label{eq70}
B'^*(\hat{u}, \hat{u})\lambda \in L^2(D(A)').
\end{equation}
Then, for all $w\in \mathbb{W}_0$ we can rewrite (\ref{ad-5}) as the
following equality
\begin{eqnarray*}
(\Delta_\alpha\lambda,w_t)_{L^2(D(A)'),L^2(D(A))}&=&\langle -\nu
A\Delta_\alpha \lambda - B'^*(\hat{u}, \hat{u})\lambda -
\gamma_u\Delta
A(\hat{u}-u_d),w\rangle_{L^2(D(A)'),L^2(D(A))}\\
&&+\gamma_T(w(T),\hat{u}(T)-u_T).
\end{eqnarray*}
Since $w(T)$ is  arbitrary, as
$\Delta_\alpha\lambda(T)=\gamma_T(\hat{u}(T)-u_T),$ we have the
existence of a representation of $\Delta_\alpha \lambda_t$ in a
distributional sense as being
$$\Delta_\alpha \lambda_t=\nu A\Delta_\alpha
\lambda+B'^*(\hat{u}, \hat{u})\lambda+\gamma_1 \Delta
A(\hat{u}-u_d).$$ 
Thus we obtain that $\lambda \in L^2(D(A))$ is a
solution of system
\begin{equation}
\left\{
\begin{array}{rcl}
\Delta_\alpha\lambda_t-\nu A\Delta_\alpha \lambda-B'^*(\hat{u},
\hat{u})\lambda&=&\gamma_u \Delta A(\hat{u}-u_d)\ \mbox{in}\ L^2(D(A)'),\label{e17}\\
\Delta_\alpha\lambda(T)&=&\gamma_T(\hat{u}(T)-u_T).
\end{array}
\right.
\end{equation}
Moreover, from (\ref{ad-4}), (\ref{eq61}) and (\ref{eq70}) we have
\begin{eqnarray}\label{eq75}
\langle B'^*(\hat{u}, \hat{u})\lambda,w\rangle_{D(A)',D(A)}&=&-(\hat{u}\cdot\nabla \lambda,w)
+\alpha^2(\hat{u}\cdot\nabla \lambda,\Delta w)
-\alpha^2(\lambda \cdot\nabla \hat{u}, \Delta h)\nonumber \\
&&-(w\cdot\nabla \lambda, \Delta_\alpha \hat{u}) +\alpha(\lambda\cdot\nabla \Delta\hat{u},w).
\end{eqnarray}
Observing that $ v\cdot \nabla h =0$ on $\Gamma$ if $v, h\in D(A)$,
and using integration by parts on $\Omega$, for $\lambda, w \in D(A)$ we obtain
\begin{eqnarray}\label{eq76}
\alpha^2(\hat{u}\cdot\nabla \lambda,\Delta w)- \alpha^2 (\lambda\cdot\nabla \hat{u}, \Delta w) &=&-\alpha^2(\nabla(\hat{u}\cdot\nabla
\lambda),\nabla w) +\alpha^2(\nabla(\lambda\cdot\nabla\hat{u}),\nabla w)\nonumber\\
&=&\alpha^2(\Delta(\hat{u}\cdot\nabla\lambda),w)-\alpha^2(\Delta(\lambda\cdot\nabla\hat{u}), w).
\end{eqnarray}
Taking into account (\ref{eq7}), we have
\begin{equation}\label{eq77}
-(w\cdot\nabla \lambda, \Delta_\alpha \hat{u}) =\langle w\cdot\nabla
\Delta_\alpha \hat{u},\lambda\rangle_{V',V} =-((\nabla\lambda)^*\cdot \Delta_\alpha \hat{u},w).
\end{equation}
Thus, from (\ref{eq75})-(\ref{eq77}) we obtain
\begin{eqnarray*}
\langle B'^*(\hat{u}, \hat{u})\lambda, w\rangle_{D(A)',D(A)}
&=&\langle-\hat{u}\cdot\nabla \lambda+\alpha^2\Delta(\hat{u}\cdot\nabla\lambda)
 -\alpha^2 \Delta(\lambda\cdot\nabla \hat{u}), w\rangle_{D(A)',D(A)}\\
&&-\langle (\nabla \lambda)^*\cdot \Delta_\alpha \hat{u}
+\alpha^2\lambda \cdot\nabla \Delta\hat{u},w\rangle_{D(A)',D(A)},
\end{eqnarray*}
which implies that the following equality as sense in $L^2(D(A))'$
\begin{equation}\label{eq78}
B'^*(\hat{u}, \hat{u})\lambda=-\hat{u}\cdot\nabla \lambda
+\alpha^2\Delta(\hat{u}\cdot\nabla \lambda) -\alpha^2(\lambda\cdot\nabla
\hat{u}) -(\nabla \lambda)^*\cdot \Delta_\alpha \hat{u}
+\alpha^2\lambda \cdot\nabla \Delta\hat{u}.
\end{equation}
Consequently, from (\ref{e17}) and (\ref{eq78}), we deduce system (\ref{adj-1}).

\end{proof}
Summarizing the state equation (\ref{eq17}), the adjoint equation
(\ref{adj-1}) and the optimality condition (\ref{adj-2-2}) we get the
optimality system.

\begin{remark}
Since $\mathcal{U}$ is a closed and convex set in the Hilbert spaces $L^2(Q)$, then from optimality condition (\ref{adj-2-2}) and \cite[Theorem 5.2, p. 132]{brez} we deduce that
the control $\hat{f}$ can be characterized as the projection of Lagrange multiplier $\lambda$ onto $\mathcal{U}$, that is,
\begin{equation}\label{proj}
\hat{f}=\mathop{\rm Proj}\limits_{\mathcal{U}}\left(-\frac{1}{\gamma_f}\lambda\right)\ \mbox{ a.e. in }\ Q.
\end{equation}
\end{remark}

\section{Relationship between the optimality systems of Navier-Stokes-$\alpha$ and Navier-Stokes models}
In \cite{Casas2}, the authors studied a velocity tracking control problem associated with the non-stationary Navier-Stokes equations for three-dimensional flows. In the classical tracking control problem, the cost functional involves the $L^2$-norm of 
$u-u_d,$ but unlike the $2D$ case, the $3D$ version is much more complicated due to the lack of uniqueness of weak solutions, or the existence of strong solutions. Therefore, instead of considering the $L^2$-norm of the cost functional, in \cite{Casas2} the authors considered
\begin{equation}\label{casas_func}
J_0(u,f):=\displaystyle\frac{\gamma_u}{2}\int_0^T\|u(t)-u_d(t)\|^8_{L^4}dt+\displaystyle\frac{\gamma_T}{2}\int_\Omega|u(x,T)-u_T(x)|^2dx
+\displaystyle\frac{\gamma_f}{2}\int_0^T\|f(t)\|^2dt.
\end{equation}
Then, it is possible to minimize $J_0$ in a class of functions which $(u,f)$ satisfies the Navier-Stokes system (\ref{eq0a}). Indeed, if $u$ is a weak solution of (\ref{eq0a}) such that $J_0(u,f)<\infty,$ then $u$ is a strong solution. With this formulation, the authors in \cite{Casas2} proved that there exists an optimal solution and analyzed first and second optimality conditions.

In this section, we are interested in to analyze the convergence of the optimality system of the optimal control problem associated to the Navier-Stokes-$\alpha$ system as $\alpha\rightarrow 0^+,$ and relate the limit to the corresponding optimality system of the optimal control problem with state equations (\ref{eq0a}) and cost functional (\ref{casas_func}). For that, we consider the following optimal control problem associated to the Navier-Stokes-$\alpha$ system:
\begin{equation}\label{eq25}
\left\{
\begin{array}{l}
\mbox{Find $(u,f)\in \mathbb{W}_{u_0}\times \mathcal{U}$ such that:}\\
J_0(u,f):=\displaystyle\frac{\gamma_u}{2}\int_0^T\|u(t)-u_d(t)\|^8_{L^4}+\displaystyle\frac{\gamma_T}{2}\int_\Omega|u(x,T)-u_T(x)|^2dx+\displaystyle\frac{\gamma_f}{2}\int_0^T\|f(t)\|^2dt\\
\mbox{is minimized, subject to $(u,f)$ being a weak solution of  (\ref{eq11}).}
\end{array}
\right.
\end{equation}
As in Section 3, the pair $(u_d,u_T)\in D(A)\times H$ represents the desires states and the nonnegative real numbers $\gamma_u$, $\gamma_T$ and $\gamma_f$ measure 
the cost of the states and control, respectively. These numbers are non zero simultaneously. The functional $J_0$ describes the
deviation of the velocities field from a desired field $u_d$, and the deviation of the velocities field in the final time $T$ from a desired field $u_T$, plus
the cost of the control $f$ measured in the $L^2$-norm.

In \cite{lans7} the authors investigated the convergence, as $\alpha\rightarrow 0^+,$ of the solutions of the Navier-Stokes-$\alpha$ equations to a weak solution of the Navier-Stokes equations (\ref{eq0a}). Here, we will analyze the convergence, as $\alpha\rightarrow 0^+,$  of the adjoint system associated to the optimal control problem (\ref{eq25}) and its relation with the corresponding adjoint system in the case of Navier-Stokes model established in \cite{Casas2}.

Following the same arguments provided in Sections 3 and 4, we get the following results:
\begin{theorem}\label{existence-optimal_new}
Let $u_0\in V$. We assume that either $\gamma_f>0$ or $\mathcal{U}$ is bounded in $L^2(Q).$ Then the extremal problem (\ref{eq25}) has at least one global optimal solution $(\hat{u},\hat{f})\in\mathcal{S}_{ad}$.
\end{theorem}
\begin{theorem}\label{mult2}
Let $\hat{x}=(\hat{u},\hat{f})\in\mathcal{S}_{ad}$ be a local optimal solution for problem (\ref{p-6}). Then, there exists
a Lagrange multiplier $\lambda\in L^2(D(A))$ which satisfy the adjoint system
\begin{equation}\label{adj-12}
\left\{
\begin{array}{rcl}
-\Delta_\alpha\lambda_t+\nu\Delta_\alpha A\lambda-\hat{u}\cdot\nabla\lambda&+&\alpha^2\Delta(\hat{u}\cdot\nabla\lambda+\lambda\cdot\nabla\hat{u})
-(\nabla\lambda)^*\Delta_\alpha\hat{u}+\alpha^2\lambda\cdot\nabla(\Delta\hat{u})\\
&=&-\gamma_u\Vert \hat{u}-u_d\Vert^4_{L^4}\vert \hat{u}-u_d\vert(\hat{u}-u_d)\ \mbox{ in }\ L^2(D(A)'),\\
{\rm div}\,\lambda&=&0\ \mbox{ in }\ Q,\\
\lambda&=&0\ \mbox{ on }\ \Gamma\times(0,T),\\
\Delta_\alpha\lambda(T)&=&\gamma_T(\hat{u}(T)-u_T)\ \mbox{ in }\ \Omega,
\end{array}
\right.
\end{equation}
and the optimality condition
\begin{equation}\label{adj-2}
\int_0^T(\gamma_f\hat{f}+\lambda,f-\hat{f})\ge 0\ \forall f\in\mathcal{U}.
\end{equation}
\end{theorem}
\begin{proof}
The proof follows the same spirit of the proof of Theorem \ref{adj-1}, noting that the functional $J_0$ is Fr\'echet differentiable and the Fr\'echet drivative of $J_0$ in  
$\hat{x}=(\hat{u},\hat{f})\in\widetilde{\mathbb{X}}:=\mathbb{W}_{u_0}\times L^2(Q)$ in
the direction ${r}=(w,z) \in \widetilde{\mathbb{X}}$ is 
\begin{equation}\label{p-8}
J_0'(\hat{x})[r]=\gamma_u\int_0^T (\Vert \hat{u}-u_d\Vert^4_{L^4}\vert \hat{u}-u_d\vert(\hat{u}-u_d),w)dt+\gamma_T(w(T),\hat{u}(T)-u_T)
+\gamma_f\int_0^T(\hat{f},z)dt.
\end{equation}
\end{proof}
Now we derive some uniform estimates of the solution of the adjoint system (\ref{adj-12}). For that, testing (\ref{adj-12})$_1$ by $\lambda$,  using the H\"older, Young and interpolation inequalities, we get:
\begin{eqnarray}\label{ee1}
\langle -\Delta_\alpha\lambda_t,\lambda\rangle_{D(A)',D(A)}
&=&-(\lambda_t,\lambda)+\alpha^2(\Delta \lambda_t,\lambda)=-\frac{1}{2}\frac{d}{dt}\left\{ \Vert \lambda\Vert^2+\alpha^2\Vert \nabla\lambda\Vert^2\right\},\label{ee2}\\
\langle \nu\Delta_\alpha A\lambda,\lambda\rangle_{D(A)',D(A)}
&=&\nu(A\lambda,\lambda)-\nu\alpha^2\langle\Delta A\lambda,\lambda\rangle_{D(A)',D(A)}=-\nu\Vert \nabla \lambda\Vert^2-\nu\alpha^2\Vert A\lambda\Vert^2,\\
\alpha^2\langle\Delta(\hat{u}\cdot\nabla\lambda),\lambda\rangle_{D(A)',D(A)}&=&\alpha^2(\hat{u}\cdot\nabla\lambda,A\lambda)\leq \alpha^2\Vert \hat{u}\Vert_{L^6}\Vert \nabla \lambda\Vert_{L^3}\Vert A\lambda\Vert\nonumber\\
&\leq&\epsilon \Vert A\lambda\Vert^2+C_\epsilon\alpha^4\Vert \nabla \hat{u}\Vert^2\left[ \Vert \nabla \lambda\Vert\Vert A\lambda\Vert+\Vert \nabla \lambda\Vert^2\right]\nonumber\\
&\leq&2\epsilon \Vert A\lambda\Vert^2+C_\epsilon\alpha^8\Vert \nabla \hat{u}\Vert^4\Vert \nabla \lambda\Vert^2+C_\epsilon\alpha^4\Vert \nabla \hat{u}\Vert^2\Vert \nabla \lambda\Vert^2\label{ee3},\\
\alpha^2\langle\Delta(\lambda\cdot\nabla\hat{u}),\lambda\rangle_{D(A)',D(A)}&\leq &\alpha^2\Vert \lambda\Vert_{L^4}\Vert\nabla \hat{u}\Vert_{L^4}\Vert A\lambda\Vert\leq \epsilon \Vert A\lambda\Vert^2+C_\epsilon\alpha^4\Vert \lambda\Vert^2_{L^4}\Vert\nabla \hat{u}\Vert^2_{L^4}\nonumber\\
&\leq & \epsilon \Vert A\lambda\Vert^2+C_\epsilon\alpha^4\left [ \Vert \lambda\Vert^{1/2}\Vert \nabla\lambda\Vert^{3/2}+\Vert \lambda\Vert^2\right]\Vert\nabla \hat{u}\Vert^2_{L^4}\nonumber\\
&\leq & \epsilon \Vert A\lambda\Vert^2+C_\epsilon\alpha^4\left [ \Vert \lambda\Vert^{1/2}\Vert \nabla\lambda\Vert^{3/2}+\Vert \lambda\Vert^2\right]\Vert\nabla \hat{u}\Vert^2_{L^4}\nonumber\\
&\leq & \epsilon \Vert A\lambda\Vert^2+C_\epsilon\alpha^4\left [ \Vert \lambda\Vert^{2}\Vert \Vert\nabla \hat{u}\Vert^2_{L^4}+\Vert \nabla \lambda\Vert^2\Vert\nabla \hat{u}\Vert^2_{L^4}\right],\label{ee4}
\end{eqnarray}
\begin{eqnarray}
&&-\langle(\nabla\lambda)^*\Delta_\alpha\hat{u},\lambda\rangle_{D(A)',D(A)}+\alpha^2\langle\lambda\cdot\nabla(\Delta\hat{u}),\lambda\rangle_{D(A)',D(A)}\nonumber\\
&&=-\langle(\nabla\lambda)^*\hat{u},\lambda\rangle_{D(A)',D(A)}
+\alpha^2\langle(\nabla\lambda)^*\Delta\hat{u},\lambda+\alpha^2\langle\lambda\cdot\nabla(\Delta\hat{u}),\lambda\rangle_{D(A)',D(A)}\nonumber\\
&&=-\langle(\nabla\lambda)^*\hat{u},\lambda\rangle_{D(A)',D(A)},
\end{eqnarray}
\begin{eqnarray}\label{ee6}
-\langle(\nabla\lambda)^*\hat{u},\lambda\rangle_{D(A)',D(A)}&\leq & \Vert \nabla \lambda\Vert \Vert \hat{u}\Vert_{L^4}\Vert \lambda\Vert_{L^4}
\leq \Vert \nabla \lambda\Vert \Vert \hat{u}\Vert_{L^4}\left[ \Vert \lambda\Vert^{1/4}\Vert \nabla \lambda\Vert^{3/4}+\Vert \lambda\Vert^{2}\right]\nonumber\\
&\leq& \epsilon\Vert \nabla \lambda\Vert ^2+C_\epsilon\Vert \hat{u}\Vert^8_{L^4} \Vert \lambda\Vert^{2}+C_\epsilon\Vert \hat{u}\Vert^2_{L^4} \Vert \lambda\Vert^{2}.
\end{eqnarray}
Collecting the estimates (\ref{ee1})-(\ref{ee6}), and denoting by $\lambda^\alpha$ the solution $\lambda$ of (\ref{adj-12}) with parameter $\alpha,$ we can conclude the following uniform estimates with respect to parameter $\alpha:$
\begin{equation}\label{ee7}
\Vert \lambda^\alpha\Vert_{L^\infty(L^2)}+\alpha^2\Vert \nabla\lambda^\alpha\Vert_{L^\infty(L^2)}\leq C\ \mbox{ and }\ 
\Vert \nabla\lambda^\alpha\Vert_{L^2(L^2)}+\alpha^2\Vert A\lambda^\alpha\Vert_{L^2(L^2)}\leq C.
\end{equation}
Using (\ref{ee7}) and following the same argument used to get (\ref{bo-9}) we also obtain that
\begin{eqnarray}\label{ee9}
\Vert \lambda^\alpha_t\Vert_{L^2(L^2)}\leq C.
\end{eqnarray}
Previous estimates imply that there exists a subsequence $\{\lambda^{\alpha_j}\}_{\alpha_j>0}$ of $\{ \lambda^\alpha\}_{\alpha>0},$ and a corresponding function $\tilde\lambda$  such that:
\begin{eqnarray*}
\lambda^{\alpha_j}&\rightarrow& \tilde\lambda\ \mbox{weakly in}\ L^2(V)\ \mbox{as}\ \alpha_j\rightarrow 0^+,\\
\lambda_t^{\alpha_j}&\rightarrow& \tilde\lambda_t\ \mbox{weakly in}\ L^2(H)\ \mbox{as}\ \alpha_j\rightarrow 0^+,\\
A\lambda^{\alpha_j}&\rightarrow& A\tilde\lambda\ \mbox{weakly in}\ L^2(H)\ \mbox{as}\ \alpha_j\rightarrow 0^+.
\end{eqnarray*}
By virtue of the above convergences, it is straightforward to see that 
\begin{eqnarray*}
-\Delta_{\alpha_j}\lambda^{\alpha_j}_t+\nu\Delta_{\alpha_j} A\lambda^{\alpha_j}-\hat{u}\cdot\nabla\lambda^{\alpha_j}&\rightarrow&  -\tilde\lambda_t+\nu A\tilde\lambda \ \mbox{weakly in}\ L^2(D(A)^\prime)\ \mbox{as}\ \alpha_j\rightarrow 0^+,\\
\alpha_j ^2\Delta(\hat{u}\cdot\nabla\lambda_{\alpha_j} +\lambda_{\alpha_j} \cdot\nabla\hat{u})+\alpha_j^2\lambda_{\alpha_j}\cdot\nabla(\Delta\hat{u})
&\rightarrow& \hat{u}\cdot\nabla\tilde\lambda-(\nabla\tilde\lambda)^*\hat{u} \ \mbox{weakly in}\ L^2(D(A)^\prime)\ \mbox{as}\ \alpha_j\rightarrow 0^+.
\end{eqnarray*}
Consequently, we obtain, as $\alpha_j\rightarrow 0^+,$ the adjoint system associated to the optimal control problem for the Navier-Stokes model:
\begin{equation}\label{adj-12j}
\left\{
\begin{array}{rcl}
-\tilde\lambda_t+\nu A\tilde\lambda-\hat{u}\cdot\nabla\tilde\lambda
-(\nabla\tilde\lambda)^*\hat{u}&=&-\gamma_u\Vert \hat{u}-u_d\Vert^4_{L^4}\vert \hat{u}-u_d\vert(\hat{u}-u_d),\\
{\rm div}\,\tilde\lambda&=&0\ \mbox{ in }\ Q,\\
\tilde\lambda&=&0\ \mbox{ on }\ \Gamma\times(0,T),\\
\tilde\lambda(T)&=&\gamma_T(\hat{u}(T)-u_T)\ \mbox{ in }\ \Omega.
\end{array}
\right.
\end{equation}

{\bf Acknowledgments:} E. Mallea-Zepeda was supported by Proyecto UTA-Mayor 4743-19, Universidad de Tarapac\'a.
E.J. Villamizar-Roa has been supported by Vicerrectoría de Investigación y Extensión of Universidad Industrial de Santander, and Fondo
 Nacional de Financiamiento para la Ciencia, la Tecnología y la Innovación Francisco José de Caldas, contrato Colciencias FP 44842-157-2016.
E. Ortega-Torres was supported by Fondecyt-Chile, Grant 1080399.

\end{document}